\documentclass[10pt]{article}
\usepackage{amssymb,mathrsfs,amsmath,amsthm}

\newcommand{\twostac}[2]{\left(\begin{array}{c}{#1}\\{#2}\end{array}\right)}

\let\a=\alpha
\let\b=\beta

\let\d=\delta

\let\e=\varepsilon
\let\f=\varphi
\let\F=\Phi
\let\g=\gamma
\let\G=\Gamma

\let\i=\iota
\let\k=\kappa
\let\l=\lambda
\let\L=\Lambda
\let\m=\mu
\let\n=\nu
\let\N=\nabla

\let\r=\rho

\let\S=\Sigma

\let\w=\omega

\let\u=\theta
\let\y=\psi

\def\bbC{{\mathbb C}}
\def\bbD{{\mathbb D}}

\def\bbT{{\mathbb T}}
\def\bbV{{\mathbb V}}

\def\cL{{\mathcal{L}}}

\def\cR{{\mathcal{R}}}

\def\cV{{\mathcal{V}}}

\def\gg{{\mathfrak{g}}}

\let\da=\downarrow

\let\ua=\uparrow
\let\os=\oplus
\let\ot=\otimes

\def\endo{\operatorname{End}}
\def\eye{\sqrt{-1}}

\def\id{\operatorname{id}}

\def\spin{\operatorname{Spin}}

\def\sym{\operatorname{Sym}}

\def\dc{\hbox{$\nabla\mkern-12mu$\raisebox{0.3ex}{$/$}$\;$}}

\def\congn{\cong_{\spin(n)}}

\let\ptl=\partial

\def\sideremark#1{\ifvmode\leavevmode\fi\vadjust{\vbox to 0pt{\vss
 \hbox to 0pt{\hskip\hsize\hskip1em
 \vbox{\hsize2cm\tiny\raggedright\pretolerance10000
 \noindent #1\hfill}\hss}\vbox to8pt{\vfil}\vss}}}%

\def\Lap{\triangle}

\newtheorem{thm}{Theorem}[section]
\newtheorem{rmk}[thm]{Remark}
\newtheorem{cor}[thm]{Corollary}
\newtheorem{lem}[thm]{Lemma}

\title{Spectra of Higher Spin Operators on the Sphere}
\author{Doojin Hong}

\date

\begin{document}

\maketitle
\begin{abstract}
We present explicit formulas for the spectra of higher spin operators (\cite{BSSV:02}) on the subbundle of the bundle of spinor-valued trace free symmetric tensors that are annihilated by the Clifford multiplication over the standard sphere in odd dimension. In even dimensional case, we give the spectra of the square of such operators. The Dirac and Rarita-Schwinger operators are zero-form and one-form cases, respectively. We also give eigenvalue formulas for the conformally invariant differential operators of all odd orders on the subbundle of the bundle of spinor-valued forms that are annihilated by the Clifford multiplication in both even and odd dimensions on the sphere.  
\end{abstract}
\section{Introduction}
The higher spin operators are generalized gradients like the Dirac and Rarita-Schwinger operators (\cite{BSSV:02}, \cite{SW:68}). They are defined on the subbundle of the bundle of spinor-valued trace free symmetric tensors that are annihilated by the Clifford multiplication ((\ref{CM})). On the standard sphere $S^n$ with $n$ odd, they act as a constant on each $\spin(n+1)$ irreducible summand of the section space. We apply the spectrum generating technique (\cite{BOO:96}) to get the eigenvalue quotients between $\spin(n+1)$ summands and thus the spectral function. Then Theorem \ref{DB} and an easy computation of the eigenvalue on a single $\spin(n+1)$ summand lead us to the complete eigenvalue formulas of the operators. In even dimensional case, these operators map positive spinors to negative spinors and vice versa. So we consider the square of the operators and get eigenvalue formulas. 
\par
Similarly, we consider the spin operators on the subbundle of spinor-forms that are annihilated by the Clifford multiplication. In this setting, only two different $\spin(n+1)$ isotypic types ((\ref{decomp})) appear and we give eigenvalue formulas for all odd order conformally invariant differential operators in both even and odd dimensions.      
\section{Conformally covariant operators}
Let $(M,g)$ be an $n$-dimensional pseudo-Riemannian manifold. 
If $f$ is a (possibly local) diffeomorphism on $M$, we denote by $f\cdot$ the 
natural action of $f$ on tensor fields which acts on vector fields as 
$f\cdot X=(df)X$ and on covariant tensors as $f\cdot \phi=(f^{-1})^*\phi$.  

A vector field $T$ is said to be {\it conformal} with {\it conformal factor} 
$\omega\in C^{\infty}(M)$ 
if 
$$
\mathcal{L}_T\, g=2\omega g\, ,
$$ 
where $\mathcal{L}$ is the Lie derivative and $g$ is the metric tensor. The conformal vector fields form a 
Lie algebra $\mathfrak{c}(M,g)$.
A conformal transformation on $(M,g)$ is a (possibly local) diffeomorphism $h$ 
for which $h\cdot g=\Omega^2 g$ for some positive function $\Omega\in 
C^{\infty}(M)$. The global conformal transformations form a group $\mathscr{C}(M,g)$. 
Let $\mathscr{T}$ be a space of $C^\infty$ tensor fields of some fixed type over $M$. For example, we can take 2-forms or trace-free symmetric covariant three-tensors. 
We have representations (\cite{Branson:96}) defined by 
\begin{align}\label{int}
\mathfrak{c}(M,g)\stackrel{U_a}{\longrightarrow} \mbox{ End }\mathscr{T},\quad
&U_a(T)=\mathcal{L}_T+a\omega\, \mbox{ and}\\
\mathscr{C}(M,g)\stackrel{u_a}{\longrightarrow} \mbox{ Aut }\mathscr{T}, \quad
&u_a(h)=\Omega^a h\cdot \nonumber
\end{align}
for $a\in \bbC$.

Note that if a conformal vector field $T$ integrates to a one-parameter group of 
global conformal transformation $\{h_{\e}\}$, then
$$
\{U_a(T)\phi\}(x)=\frac{d}{d\e}\Big|_{\e=0}\{u_a(h_{-\e})\phi\}(x)\, .
$$
In this sense, $U_a$ is the infinitesimal representation corresponding to $u_a$. 
\\
A differential operator $D:C^{\infty}(M)\rightarrow C^{\infty}(M)$ is said to be 
{\it infinitesimally conformally covariant of bidegree} $(a,b)$ if 
$$
DU_a(T)\phi=U_b(T)D\phi
$$
for all $T\in \mathfrak{c}(M,g)$ and $D$ is said to be {\it conformally covariant of bidegree} $(a,b)$ if 
$$
Du_a(h)\phi=u_b(h)D\phi
$$
for all $h\in \mathscr{C}(M,g)$. 

To relate conformal covariance to conformal invariance, we recall that 
conformal weight of a bundle $V$ with the induced bundle metric $g_{{}_V}$ from $g$ is $r$ 
iff 
$$
\widehat{g}=\Omega^2 g \Longrightarrow \widehat{g}_{{}_V}=\Omega^{-2r} g_{{}_V}\, .
$$
Tangent bundle, for instance, has conformal weight -1. Let us denote a bundle $V$ with 
conformal weight $r$ by $V^r$. Then we can impose new conformal weight $s$ on $V^r$ by 
taking tensor product of it with the bundle $I^{(s-r)/n}$ of scalar $((s-r)/n)$-densities (\cite{Branson:97}). Now if 
we look at an operator of bidegree $(a,b)$ as an operator from the bundle with conformal weight $-a$ to the bundle with conformal weight $-b$, the operator becomes conformally 
invariant.

As an example, let us consider the conformal Laplacian on $M$: 
\begin{equation*}
Y=\triangle+\frac{n-2}{4(n-1)}\text{Scal}, 
\end{equation*}
where $\triangle=-g^{ab}\nabla_a\nabla_b$ and Scal is the scalar curvature. 
Note that $Y: C^\infty(M) \rightarrow C^\infty(M)$ is conformally covariant of bidegree $((n-2)/2,(n+2)/2)$. That is,
\begin{equation*}
\widehat{Y}=\Omega^{-\frac{n+2}{2}}Y\mu(\Omega^{\frac{n-2}{2}})\, ,
\end{equation*}
where $\widehat{Y}$ is $Y$ evaluated in $\widehat{g}$ and 
$\mu(\Omega^{\frac{n-2}{2}})$ is multiplication by $\Omega^{\frac{n-2}{2}}$.
If we let $V=C^\infty(M)$ and view $Y$ as an operator 
\begin{equation*}
Y: V^{-\frac{n-2}{2}} \rightarrow V^{-\frac{n+2}{2}}\, ,
\end{equation*}
we have, for $\phi\in V^{-\frac{n-2}{2}}$, 
\begin{equation*}
\widehat{Y}\,\,\widehat{\phi}=\widehat{Y\phi}\, ,
\end{equation*}
where $\widehat{Y}$, $\widehat{\phi}$, and $\widehat{Y\phi}$ are $Y$, $\phi$, and $Y\phi$ computed in $\widehat{g}$, respectively.
\section{Dominant weights}\label{VB}
Let $\lambda$ be a dominant weight of irreducible Spin($n$) representation. That is, 
$$\lambda=(\l_1,\ldots,\l_l)\in \mathbb{Z}^l\cup(1/2+\mathbb{Z})^l,\,\quad l=[n/2]
$$ 
satisfying the inequality constraint (dominant condition)
\begin{equation*}
\begin{array}{ll}
\l_1\ge\ldots\ge\l_l\ge0,&n\text{ odd},\\
\l_1\ge\ldots\ge\l_{l-1}\ge |\l_l|,&n\text{ even}.
\end{array}
\end{equation*}
$\l$ is identified with the highest weight of the irreducible representation of Spin($n$)  (\cite{IT:78}). We shall denote by $V(\l)$ the representation with the highest weight $\l$. Those $\lambda\in\mathbb{Z}^l$ are exactly the representations that factor through SO($n$). For example, $V(1,0,\dots,0)$ and $V(1,1,1,0,\dots,0)$ are the defining representation and the three-form representation of SO($n$), respectively and $V(\frac{1}{2},\dots,\frac{1}{2})$ is the spinor representation in odd dimensional case.

If $M$ is an $n$-dimensional smooth manifold with Spin($n$) structure and $\mathcal{F}$ is the bundle of spin frames, we denote by $\mathbb{V}(\l)$ the associated vector bundle $\mathcal{F}\times_\l V(\l)$. 
\section{Intertwining relation}
Let $G=\text{Spin}_0(n+1,1)$ be the identity component of the $\text{Spin}(n+1,1)$ and $\mathfrak{g}=\mathfrak{k}+\mathfrak{s}$ be a Cartan decomposition of the Lie algebra $\mathfrak{g}$ of $G$. Then, in an Iwasawa decomposition $G=KAN$, the maximal compact subgroup $K$ of $G$ is a copy of $\text{Spin}(n+1)$. Let $M$ be the centralizer of the Lie algebra $\mathfrak{a}$ of $A$ in $K$. Then $M$ is a copy of $\text{Spin}(n)$ and $P=MAN$ is a maximal parabolic subgroup of $G$. Note that $G/P=K/M$ is diffeomorphic to the sphere $S^n$ (\cite{BOO:96}). 

Let $V(\lambda)$ be a finite dimensional irreducible representation of $M$. Consider the $G$ module $\mathcal{E}(G;\lambda,\nu)$ of $C^\infty$ functions
\begin{equation*}
F: G\rightarrow V(\lambda) \text{ with }F(gman)=a^{-\nu-\rho}\lambda(m)^{-1}F(g), \quad g\in G, m\in M, a\in A, n\in N,
\end{equation*}
where $\rho$ is half the sum of the positive $(\mathfrak{g},\mathfrak{a})$ roots. This space is in one-to-one correspondence with the space of smooth sections of $\mathbb{V}(\lambda)$, the $K$ module $\mathcal{E}(K;\lambda|_{K\cap M})$ of $C^\infty$ functions
\begin{equation*}
f:K\rightarrow V(\lambda) \text{ with }f(km)=\lambda(m)^{-1}f(k),\quad k\in K, m\in M.
\end{equation*}
The $K$-finite subspace $\mathcal{E}_K(G;\lambda,\nu)\cong_K \mathcal{E}_K(K;\lambda|_{K\cap M})$ is defined as 
\begin{equation}\label{decomp}
\bigoplus_{\alpha\in \hat{K},\, \alpha \downarrow \lambda}\mathcal{V}(\alpha),
\end{equation}
where $\hat{K}$ is the set of dominant Spin($n+1$) weights and $\mathcal{V}(\alpha)$ is the $\alpha$-isotypic component satisfying the classical branching rule of $K$ and $M$ (\cite{Branson:92}):
\begin{equation*}
\alpha\downarrow\lambda\text{ iff }\alpha_1-\lambda_1\in \mathbb{Z}\text{ and }
\begin{cases} \alpha_1\ge\lambda_1\ge\alpha_2\ge \cdots \ge\lambda_l\ge |\alpha_{l+1}|, & \text{n odd}\\
\alpha_1\ge\lambda_1\ge\alpha_2\ge \cdots \ge\lambda_{l-1}\ge \alpha_l\ge |\lambda_l|, &\text{n even}.
\end{cases}
\end{equation*}
The conformal action of $G$ and its infinitesimal representation correspond to those in (\ref{int}).

Let $A=A_{2r}$ be an intertwinor of order $2r$ of the $(\gg,K)$ representation. If $X\in\gg$ with its conformal factor $\w$, $A$ is a $K$-map satisfying
the intertwining relation
\begin{equation}\label{Int-rel}
A\left(\tilde{\cL}_X+\left(\frac{n}2-r\right)\omega\right)=
\left(\tilde{\cL}_X+\left(\frac{n}2+r\right)\omega\right)A,
\end{equation}
where $\tilde{\cL}_X$ is the {\em reduced Lie derivative} (\cite{Branson:97}), $\tilde{\cL}_X=\cL_X+(l-m)\omega$ on tensors of $\twostac{l}{m}$-type ($l$ contravariant and $m$ covariant). 

\section{Higher spin operators}\label{Hi}
Let $\bbV(\l)$ be an irreducible associated vector bundle (\ref{VB}). The covariant derivative takes sections of $\bbV(\l)$ to sections of (\cite{Branson:92}) 
\begin{equation}\label{cov}
T^*M\ot\bbV(\l)\congn\bbV(\m_1)\os\cdots\os\bbV(\m_N),
\end{equation} 
where
\begin{align*}
&\m_i=\l\pm e_a, \text{ for some }a\in\{1,\ldots,[n/2]\}\\
&\text{ or }\\
&\m_i=\l\, \text{ if $n$ is odd},\, \l_{[n/2]}\ne 0 .
\end{align*}
Here $e_a=(0,\ldots,0,1,0,\ldots,0)$ with ``1" in the $a$-th slot.

Let us consider the fundamental tensor-spinor (\cite{BH:02})
\begin{equation}\label{FTS}
\g:T^*S^n \to \endo(\S, \S)
\end{equation}
satisfying the Clifford relation
\begin{equation}\label{CR}
\g_i\g_j+\g_j\g_i=-2g_{ij}\id_\S,
\end{equation}
where $\S$ is the spinor bundle:
\begin{align*}
&\S=\bbV(1/2,\ldots,1/2),\quad n \text{ odd},\\
&\S=\S_{+}\os\S_{-}, \quad \S_\pm=\bbV(1/2,\ldots,1/2,\pm 1/2),\quad n \text { even}.
\end{align*}

The higher spin operators are projections of the above covariant derivative action ((\ref{cov})):
\begin{align*}
&\bbV\left(\frac{1}{2}+k,\frac{1}{2},\ldots,\frac{1}{2}\right) \to \bbV\left(\frac{1}{2}+k,\frac{1}{2},\ldots,\frac{1}{2}\right), \quad n \text{ odd}\\
&\bbV\left(\frac{1}{2}+k,\frac{1}{2},\ldots,\frac{1}{2},\pm\frac{1}{2}\right)\to \bbV\left(\frac{1}{2}+k,\frac{1}{2},\ldots,\frac{1}{2},\mp\frac{1}{2}\right), \quad n \text{ even}\\
&\quad \text{for }k=1,2,\ldots ,
\end{align*}
Taking a normalization, we can express these operators as
\begin{equation}\label{HS}
(\cR^{(k)}\f)_{a_1\cdots a_k}=\g^b\N_b\f_{a_1\cdots a_k}-\frac{2k}{n+2k-2}\g_{(a_1}\N^b\f_{a_2\cdots a_k)b} ,
\end{equation} 
where indices enclosed in parentheses are symmetrized and $\f$ is a spinor-valued trace-free symmetric $k$-tensor annihilated by the Clifford multiplication. i.e. 
\begin{equation}\label{CM}
\g^{a_1}\f_{a_1a_2\cdots a_k}=0.
\end{equation}

\textbf{Case I: $\bf n$ odd}\\ \\
$\cR^{(k)}$ is an intertwinor of order 1 with $k$-covariant tensor type content. So (\ref{Int-rel}) becomes
\begin{equation*}
\cR^{(k)}\left(\cL_X+\left(\frac{n}{2}-\frac{1}{2}-k\right)\w\right)=\left(\cL_X+\left(\frac{n}{2}+\frac{1}{2}-k\right)\w\right)\cR^{(k)}.
\end{equation*}
Now we specialize to the case of $S^n$ with its standard metric. 
Let $Y=\sin\r\ptl_\r$ be a conformal vector field and $\w=\cos\r$ its corresponding conformal factor (\cite{Hong:11}).

The following lemma compares the Lie derivative and covariant derivative on the spinor-valued $k$-tensor bundle on $S^n$.
\begin{lem}\label{lem1}
For any spinor-k-tensor $\F$, 
\begin{align*}
((\cL_Y-\N_Y)\Phi)_{a_1\cdots a_k}&=k\,\omega\Phi_{a_1\cdots a_k}\, .
\end{align*}
\end{lem}
\begin{proof} Note that for a 1-form $\eta$ and a vector field $X$,
$$
\langle(\cL_Y-\N_Y)\eta,X\rangle=-\langle\eta,(\cL_Y-\N_Y)X\rangle,
$$
since $\cL_Y-\N_Y$ kills scalar functions.
But by the symmetry of the Riemannian connection,
$$
[Y,X]-\N_YX=-\N_XY.
$$
We conclude that
$$
(\cL_Y-\N_Y)\eta=\langle\eta,\N Y\rangle,
$$
where in the last expression, $\langle\cdot,\cdot\rangle$ is the pairing
of a 1-form with the contravariant part of a $\twostac{1}{1}$-tensor:
$$
((\cL_Y-\N_Y)\eta)_\l=\eta_\m\N_\l Y^\m.
$$
Since $Y$ is a conformal vector field, 
$$
(\N Y_\flat)_{\l\m}=(\N Y_\flat)_{(\l\m)}+(\N Y_\flat)_{[\l\m]}=(\omega g+\frac{1}{2}
dY_\flat)_{\l\m}=\omega g_{\l\m}, 
$$
where $(Y_\flat)_\a=g_{\a\b}Y^\b$, the contraction of $Y$ with the metric tensor $g$. Thus
\begin{equation*}
((\cL_Y-\N_Y)\eta)_\l=\omega\eta_\l\, .
\end{equation*}
Since $\cL_Y-\N_Y$ is a derivation, for any $k$-tensor $\f_{a_1\cdots a_k}$,
\begin{equation*}
((\cL_Y-\N_Y)\f)_{a_1\cdots a_k}=k\w\f_{a_1\cdots a_k}.
\end{equation*}
On the spinor bundle $\Sigma$, on the other hand, (\cite[eq(16)]{Kosmann:72})
$$
\cL_Y-\N_Y=-\tfrac14\N_{[a}Y_{b]}\g^a\g^b
=-\tfrac18(dY_\flat)_{ab}\g^a\g^b=0,
$$
where $\gamma$ is the fundamental tensor-spinor. Thus the lemma follows.
\end{proof}
The spectrum generating relation is given in the following lemma.
\begin{lem}\label{lem2}
On spinor-tensors of any type,
\begin{equation*}
[\N^*\N,\omega]=2\left(\nabla_Y+\dfrac{n}{2}\omega\right)\, ,
\end{equation*}
where [,] is the operator commutator.
\end{lem}
\begin{proof}
If $\varphi$ is any smooth section of a tensor-spinor bundle, then
\begin{equation*}
[\N^*\N,\omega]\varphi=(\Lap \omega)\varphi-2(d\omega)^\a\nabla_\a\varphi
=(n\omega+2Y^\a\N_\a)\varphi
=(n\omega+2\N_Y)\, \varphi.
\end{equation*}
\end{proof}
Thus the intertwining relation (\ref{Int-rel}) on spinor-$k$-tensors becomes
\begin{equation}\label{rel-2}
\cR^{(k)}\left(\frac{1}{2}[\N^*\N,\omega]-\frac{1}{2}\omega\right)=
\left(\frac{1}{2}[\N^*\N,\omega]+\frac{1}{2}\omega\right)\cR^{(k)}.
\end{equation}
Now we look at the $\spin(n+1)$ types in (\ref{decomp}) occurring over $\bbV(1/2+k,1/2,\ldots,1/2)$. 
Define
$$
\cV_\e(j,q):=\bbV\Big(\underbrace{\frac{1}{2}+k+j,\frac{1}{2}+q,\frac{1}{2},\ldots,\frac{1}{2},\frac{\e}{2}}_{\left[\frac{n+1}{2}\right]}\Big),\quad \e=\pm 1, \; q=0,1,\ldots,k,\; j=0,1,\ldots .
$$
For $\f\in\cV_\e(j,q)$, we have (\cite{BOO:96}) 
$$
\w\f\in\cV_\e(j+1,q)\os\cV_\e(j-1,q)\os\cV_\e(j,q+1)\os\cV_\e(j,q-1)\os\cV_{-\e}(j,q).
$$
Let $\a=\cV_\e(j,q)$ and $\b$ one of the summands in the above direct sum. We consider the compressed intertwining relation (\ref{rel-2}) between $\a$ and $\b$. 
\begin{equation*}
\m_\b\cdot\frac{1}{2}\left(\N^*\N|_\b-\N^*\N|_\a-1\right)\cdot|_\b\w|_\a\f=\frac{1}{2}\left(\N^*\N|_\b-\N^*\N|_\a+1\right)\cdot\m_\a\cdot|_\b\w|_\a\f,
\end{equation*}
where $\m_\a$, $\m_\b$ are eigenvalues of $\cR^{(k)}$ on $\a$, $\b$ isotypic summands, $|_\b\w|_\a$ is the projection of $\w$ onto $\b$ summand, and $\N^*\N|_\a$, $\N^*\N|_\b$ are evaluations of $\N^*\N$ on $\a$, $\b$ summands, respectively. 

Canceling $|_\b\w|_\a$ from both sides and computing $\N^*\N$ (\cite{Branson:92}), we get $\m_\b/\m_\a$ transition quantities. 

With respect to the diagram
\begin{equation*}
\begin{array}{ccc}
\cV_1(j+1,q) & &\cV_1(j,q+1) \\
\ua&\nearrow& \\
\boxed{\cV_1(j,q)}&\rightarrow&V_{-1}(j,q)\\
\da&\searrow& \\
\cV_1(j-1,q)& &\cV_1(j,q-1) \\
\end{array}
\end{equation*}
we get
\begin{equation}\label{DG}
\begin{array}{ccc}
(J+1)/J& &(n+2q)/(n+2q-2)\\
\ua&\nearrow& \\
\bullet&\rightarrow&-1\\
\da&\searrow& \\
(J-1)/J& &(n+2q-4)/(n+2q-2)\\
\end{array}
\end{equation}
where $J=n/2+k+j$. Thus the eigenvalue of $\cR^{(k)}$ on $\cV_1(1/2+k,1/2+k)$ completely determines the spectra of $\cR^{(k)}$.

Let $\L^k$ and $\sym^k$ be the spaces of $k$-forms and symmetric $k$-tensors, respectively.
\begin{thm}\label{DB} For $\F\in (\S\ot\L^k)\cup(\S\ot\sym^k)$ with $\g^{a_1}\F_{a_1 a_2\cdots a_k}=0$,
\begin{equation*}
\dc^2\F_{a_1\cdots a_k}=\left(\N^*\N+\frac{n(n-1)}{4}+k\right)\F_{a_1\cdots a_k},
\end{equation*}
where $\dc=\g^a\N_a$.
\end{thm}
\begin{proof}
We write $\dc^2\F=\g^a\g^b(\N_b\N_a+\cR_{ab})\F=(-\dc^2+2\N^*\N+\g^a\g^b\cR_{ab})\F$ using the Clifford relation (\ref{CR}), where $\cR_{ab}=[\N_a,\N_b]$ is the spin curvature (\cite{BH:02}). So $\dc^2\F=(\N^*\N+1/2\g^a\g^b\cR_{ab})\F$. And
$$
\g^a\g^b\cR_{ab}\F_{a_1\cdots a_k}=\g^a\g^b(W_{ab}\F_{a_1\cdots a_k}-R^\n_{\phantom{x}a_1ab}\F_{\n a_2\cdots a_k}-\cdots -R^\n_{\phantom{x}a_kab}\F_{a_1\cdots a_{k-1}\n}),
$$
where $W_{ab}=-1/4R_{\k\n ab}\g^\k\g^\n$ is the action on spinors and $R_{\k\n\l\m}$ is the Riemann curvature tensor. On $S^n$, since $R_{\k\n\l\m}=g_{\k\l}g_{\n\m}-g_{\k\m}g_{\n\l}$,
\begin{align*}
&\g^a\g^bW_{ab}=\frac{n(n-1)}{2}\\
&\g^a\g^bR^\n_{\phantom{x}a_iab}\F_{a_1\cdots a_{i-1}\n a_{i+1}\cdots a_k}=(\g^\n\g_{a_i}-\g_{a_i}\g^\n)\F_{a_1\cdots a_{i-1}\n a_{i+1}\cdots a_k}\\
&\quad =(-2\g_{a_i}\g^\n-2\d^\n_{a_i})\F_{a_1\cdots a_{i-1}\n a_{i+1}\cdots a_k}\\
&\quad =-2\F_{a_1\cdots a_k}
\end{align*}
and the theorem follows.
\end{proof}
\begin{rmk} When $k=0$, this is the classical Lichnerowicz formula on the spinor bundle over $S^n$. And, in general, $\g^a\g^bW_{ab}=\text{Scal}/2$, where Scal is the scalar curvature. 
\end{rmk}
Notice that $\cR^{(k)}$ reduces to $\dc$ on $\cV_1(1/2+k,1/2+k)$, since it is of divergence type. That is, if $\f\in\cV_1(1/2+k,1/2+k)$, then $\f=\N^*\y$ for some section $\y$ over $\bbV(1/2+k,3/2,1/2,\ldots, 1/2)$.

We can now describe the spectra of the higher spin operators.
\begin{thm} The operator $\cR^{(k)}$ acts as a constant
\begin{equation*}
\e\cdot\frac{n+2q-2}{n+2k-2}\cdot \left(\frac{n}{2}+k+j\right) \quad \text{on }\; \cV_\e\left(\frac{1}{2}+k+j,\frac{1}{2}+q\right),\; \e=\pm 1,\; q=0,1,\ldots ,k,\; j=0,1,\ldots .
\end{equation*}
\end{thm}
\begin{proof} This is a direct consequence of Theorem \ref{DB} together with the diagram of transition quantities (\ref{DG}).
\end{proof}
\begin{rmk} When $k=0,1$, we get eigenvalues of the Dirac and the Rarita-Schwinger operators, respectively (\cite{Branson:99}).
\end{rmk} 
\textbf{Case II: $\bf n$ even}\\ \\
In this case $\cR^{(k)}$ changes chirality of spinors. That is,
$$
\cR^{(k)}:\bbV\left(\frac{1}{2}+k,\frac{1}{2},\ldots,\frac{1}{2},\frac{\e}{2}\right)\to\bbV\left(\frac{1}{2}+k,\frac{1}{2},\ldots,\frac{1}{2},-\frac{\e}{2}\right),\quad \e=\pm 1.
$$
Consider (\cite{Branson:97}) the commuting diagram.
\begin{equation*}
\begin{array}{ccc}
\bbV(1/2+k,1/2,\ldots,1/2,-\e/2)&\stackrel{I_2}{\longrightarrow}&\bbV(1/2+k,1/2,\ldots,1/2,-\e/2)\\
\ua G&&\da G^*\\
\bbV(1/2+k,1/2,\ldots,1/2,\e/2)&\stackrel{I_1}{\longrightarrow}&\bbV(1/2+k,1/2,\ldots,1/2,\e/2)
\end{array}
\end{equation*}
Here $G$ is the generalized gradient (projection of the covariant derivative action, ((\ref{cov})), $G^*$ is the adjoint of $G$, and $I_1$, $I_2$ are intertwinors :
\begin{align*}
&I_1\left(\tilde{\cL}_X+\left(\frac{n-2}{2}\right)\omega\right)=
\left(\tilde{\cL}_X+\left(\frac{n+2}{2}\right)\omega\right)I_1 \text{ and}\\
&I_2\left(\tilde{\cL}_X+\left(\frac{n+2}{2}\right)\omega\right)=
\left(\tilde{\cL}_X+\left(\frac{n-2}{2}\right)\omega\right)I_2 .
\end{align*}
Since $(\cR^{(k)})^2$ is a constant multiple of $G^*G$, we just need the ratios of eigenvalues of $I_1$ to $I_2$. But $I_2$ is a constant multiple of $1/I_1$. Thus $(\cR^{(k)})^2=c\cdot (I_1)^2$ for some constant $c$.
 
The following lemma determines the constant $c$.
\begin{lem}\label{HS2} $(\cR^{(k)})^2=\dc^2$ on $\cV(1/2+k,1/2+k,1/2,\ldots,1/2)$.
\begin{proof}
Let $\f\in \cV(1/2+k,1/2+k,1/2,\ldots,1/2)$. So $\f=\N^*\y$ for some $\y$ over $\bbV(1/2+k,3/2,1/2,\ldots,\e/2)$.
Since $\cR^{(k)}\f=\dc\f$ on $\cV(1/2+k,1/2+k,1/2,\ldots,1/2)$, 
$$
(\cR^{(k)})^2\f_{a_1\cdots a_k}=\dc^2\f_{a_1\cdots a_k}-\frac{2k}{n+2k-2}\g_{(a_1}\N^b\dc\f_{a_2\cdots a_k)b} .
$$
But
\begin{align*}
&\N^b\dc\f_{a_2\cdots a_kb}=\g^c(\N_c\N_b+\cR_{bc})\f^b_{\phantom{x}a_2\cdots a_k}=\g^c\cR_{bc}\f^b_{\phantom{x}a_2\cdots a_k}\\
&\quad =\g^c(W_{bc}\f^b_{\phantom{x}a_2\cdots a_k}-R^{\n b}{}{}_{bc}\f_{\n a_2\cdots a_k}-R^\n{}_{a_2bc}\f^b{}_{\n a_3\cdots a_k}-\cdots -R^\n{}_{a_kbc}\f^b{}_{a_2\cdots a_{k-1}\n})\\
&\quad =\text{constant}\cdot\g^c\f_{ca_2\cdots a_k}\\
&\quad =0 ,
\end{align*}
where $\cR_{bc}$ is the spin curvature. Hence the claim follows.
\end{proof}
\end{lem}
We define
$$
\cV(j,q):=\bbV\Big(\underbrace{\frac{1}{2}+k+j,\frac{1}{2}+q,\frac{1}{2},\ldots,\frac{1}{2}}_{\left[\frac{n+1}{2}\right]}\Big),\quad q=0,1,\ldots,k,\; j=0,1,\ldots .
$$
For $\f\in\cV(j,q)$, we have (\cite{BOO:96}) 
$$
\w\f\in\cV(j+1,q)\os\cV(j-1,q)\os\cV(j,q+1)\os\cV(j,q-1)\os\cV(j,q).
$$
With respect to the diagram
\begin{equation*}
\begin{array}{ccc}
\cV(j+1,q) & &\cV(j,q+1) \\
\ua&\nearrow& \\
\boxed{\cV(j,q)}&&\\
\da&\searrow& \\
\cV(j-1,q)& &\cV(j,q-1) \\
\end{array}
\end{equation*}
we get
\begin{equation*}
\begin{array}{ccc}
(J+1)/J& &(n+2q)/(n+2q-2)\\
\ua&\nearrow& \\
\bullet&&\\
\da&\searrow& \\
(J-1)/J& &(n+2q-4)/(n+2q-2)\\
\end{array}
\end{equation*}
where $J=n/2+k+j$.

Thus we have
\begin{thm}\label{sq} The operator $(\cR^{(k)})^2$ acts as a constant
\begin{equation*}
\left[\frac{n+2q-2}{n+2k-2}\cdot \left(\frac{n}{2}+k+j\right)\right]^2 \quad \text{on }\; \cV\left(\frac{1}{2}+k+j,\frac{1}{2}+q\right),\; q=0,1,\ldots ,k,\; j=0,1,\ldots .
\end{equation*}
\end{thm}
\section{Spin operators over spinor-forms}
In this section, we consider the conformally invariant spin operators on $\bbV(\underbrace{3/2,\ldots,3/2}_{k},1/2,\ldots,\e/2)$ where $\e=1$ for $n$ odd, $\e=\pm 1$ for $n$ even, and $0\le k < n/2$. These operators satisfy the intertwining relation (\ref{Int-rel}). By Lemma \ref{lem1} and \ref{lem2}, we have, for an order $2r$ intertwinor $A=A_{2r}$, 
\begin{equation*}
A\left(\frac{1}{2}[\N^*\N,\omega]-r\omega\right)=
\left(\frac{1}{2}[\N^*\N,\omega]+r\omega\right)A.
\end{equation*}

\textbf{Case I: $\bf n$ odd and $\bf k=0$}\\
\\
Let $\cV_\e(j)=\cV(1/2+j,1/2,\ldots,1/2,\e/2)$ for $\e=\pm 1$.
With respect to the following diagram
\begin{equation*}
\begin{array}{ccc}
\cV_\e(j+1)&&\\
\ua&& \\
\boxed{\cV_\e(j)}&\rightarrow&V_{-\e}(j)\\
\da&& \\
\cV_\e(j-1)&&\\
\end{array}
\end{equation*}
we get
\begin{equation*}
\begin{array}{ccc}
(J+1/2+r)/(J+1/2-r)&&\\
\ua&& \\
\bullet&\rightarrow&-1\\
\da&& \\
(-J+1/2+r)/(-J+1/2-r)&&\\
\end{array}
\end{equation*}
where $J=n/2+j$. 
Choosing a normalization $\m_{{}_{\cV_\e(0)}}=\e$, we have
\begin{thm} The unique spectral function $Z_\e(r,j)$ on $\cV_\e(j)$ is up to normalization
\begin{equation*}
Z_\e(r,j)=\e\cdot\frac{\G(J+\frac{1}{2}+r)\G(\frac{n}{2}+\frac{1}{2}-r)}{\G(J+\frac{1}{2}-r)\G(\frac{n}{2}+\frac{1}{2}+r)},\quad \e=\pm 1, \quad j=0,1,2,\ldots. 
\end{equation*}
\end{thm}
If $2r=1$, $Z_\e(1/2,j)=\e\cdot J\cdot \frac{2}{n}=\frac{2}{n}\cdot\dc$ on $\cV_\e(j)$ is a constant multiple of the Dirac operator. As a consequence of the spectral function in the theorem, we get
\begin{cor} (\cite{BO:06}) The differential operator $D_{2l+1}:\S\to\S$ defined by
\begin{equation*}
D_{2l+1}:=\dc\cdot\prod_{p=1}^l (\dc^2-p^2)
\end{equation*}
is conformally invariant of order $2l+1$.
\end{cor}
\begin{proof} $D_{2l+1}$ acts as $\e\cdot J\cdot\prod_{p=1}^l (J^2-p^2)$ on $V_\e(j)$ so it is a constant multiple of $Z_\e(l+1/2,j)$.
\end{proof}
 
\textbf{Case II: $\bf n$ odd and $\bf k\ge 1$}\\
\\
Let, for $\e=\pm 1$,
\begin{equation*}
\cV_\e(j,q)=\begin{cases}\cV_\e(3/2+j,3/2,\ldots,3/2,\underbrace{1/2+q}_{(k+1)^\text{st}},1/2,\ldots,1/2,\e/2),\quad k<(n-1)/2\\
\cV_\e(3/2+j,3/2,\ldots,3/2,\e(1/2+q)),\quad k=(n-1)/2.
\end{cases}
\end{equation*}
With respect to the diagram
\begin{equation*}
\begin{array}{ccccc}
&&\cV_\e(j+1)&&\\
&&\ua&& \\
\cV_\e(j,0)&\leftarrow&\boxed{\cV_\e(j,1)}&\rightarrow&V_{-\e}(j,1)\\
&&\da&& \\
&&\cV_\e(j-1)&&\\
\end{array}
\end{equation*}
we get
\begin{equation*}
\begin{array}{ccccc}
&&(J+\frac{1}{2}+r)/(J+\frac{1}{2}-r)&&\\
&&\ua&& \\
(n-2k+1-2r)/(n-2k+1+2r)&\leftarrow&\bullet&\rightarrow&-1\\
&&\da&& \\
&&(-J+\frac{1}{2}+r)/(-J+\frac{1}{2}-r)&&\\
\end{array}
\end{equation*}
where $J=n/2+1+j$.
Choosing a normalization $\m_{{}_{\cV_\e(0,1)}}=\e$, we have
\begin{thm}\label{form} The unique spectral function $Z_\e(r,j,q)$ on $\cV_\e(j,q)$ is up to normalization
\begin{align*}
&Z_\e(r,j,q)=\e\cdot\frac{n-2k+1+2(2q-1)r}{n-2k+1+2r}\cdot\frac{\G(J+\frac{1}{2}+r)\G(\frac{n}{2}+\frac{3}{2}-r)}{\G(J+\frac{1}{2}-r)\G(\frac{n}{2}+\frac{3}{2}+r)},\\
&\quad \e=\pm 1, \quad q=0,1, \quad j=0,1,2,\ldots. 
\end{align*}
\end{thm}
\begin{rmk} If $2r=1$ and $k=1$, $Z_\e(1/2,j,q)=\e\cdot\frac{n+2(q-1)}{n}\cdot J\cdot \frac{2}{n+2}=\frac{2}{n+2}\cdot\cR^{(1)}$ on $\cV_\e(j,q)$. Here $\cR^{(1)}$ is the Rarita-Schwinger operator ((\ref{HS})).
\end{rmk}
To get all odd order conformally invariant differential operators for $k\ge 1$, we consider the following convenient operators (\cite{Branson:99}) on spinor-forms:
\begin{align*}
(\tilde{d}\varphi)_{a_0\cdots a_k} &:=\sum_{i=0}^k (-1)^i\N_{a_i}\varphi_{a_0\cdots a_{i-1}a_{i+1}\cdots a_k},\\
(\tilde{\delta}\varphi)_{a_2\cdots a_k}&:=-\N^b\varphi_{ba_2\cdots a_k},\\
(\e(\gamma)\varphi)_{a_0\cdots a_k}&:=\sum_{i=0}^k (-1)^i\gamma_{a_i}\varphi_{a_0\cdots a_{i-1}a_{i+1}\cdots a_k}, \\
(\iota(\gamma)\varphi)_{a_2\cdots a_k}&:=\gamma^b\varphi_{ba_2\cdots a_k},\\
(\mathbb{D}\varphi)_{a_1\cdots a_k}&:=(\iota(\gamma)\tilde{d}+\tilde{d}\iota(\gamma))\varphi)_{a_1\cdots a_k}=-(\tilde{\delta}\e(\gamma)+\e(\gamma)\tilde{\delta})\varphi)_{a_1\cdots a_k}=\dc\varphi_{a_1\cdots a_k}.
\end{align*}
The operator 
\begin{equation*}
P_k:=\frac{n-2k+4}{2}\iota(\gamma)\tilde{d}+\frac{n-2k}{2}(\tilde{d}\iota(\gamma)-\tilde{\delta}\e(\gamma))-\frac{n-2k-4}{2}\e(\gamma)\tilde{\delta}
\end{equation*}
on $\Sigma\ot\wedge^k$ restricted to 
$$
\bbT^k:=\begin{cases}\bbV(\underbrace{\tfrac{3}{2},\dots\tfrac{3}{2}}_{k},\frac{1}{2},\dots ,\frac{1}{2}),\quad k\ge 1\\
\bbV(\tfrac{1}{2},\ldots,\frac{1}{2}),\quad k=0\end{cases}=\{\f\in\S\ot\wedge^k \mid \g^{a_1}\f_{a_1 a_2\cdots a_k}=0\},
$$
is conformally invariant on $\bbT^k$.
\begin{rmk} $(1/(n-2k+2))\cdot P_k|_{\bbT^k}=\bbD+(2/(n-2k+2))\cdot\e(\gamma)\tilde{\delta}$ is the Dirac and Rarita-Schwinger operators when $k=0,1$, respectively. 
\end{rmk}
Since $P_k=(n-2k+2)\dc$ on $\cV_\e(j,1)$, by Theorem \ref{DB} and \ref{form},  
\begin{equation*}
P_k \text{ acts as }\quad \e\cdot (n-2k+2q)\cdot J \quad\text{ on }\cV_\e(j,q), \quad q=0,1.
\end{equation*}
Consider now the operator $T_{k-1}: \mathbb{T}^{k-1}\rightarrow\mathbb{T}^k$ defined by 
\begin{equation*}
T_{k-1}=\frac{1}{k}\tilde{d}+\frac{1}{k(n-2(k-1))}\e(\gamma)\mathbb{D}+\frac{1}{k(n-2(k-1))(n-2(k-1)+1)}\e(\gamma)^2\tilde{\delta}\, .
\end{equation*}
This is the orthogonal projection of $\nabla$ onto $\mathbb{T}^k$ summand ($1/k\cdot\tilde{d}^{\text{top}}_{k-1}$ in \cite{BH:02}):
\begin{equation*}\label{tk}
\mathbb{T}^{k-1}\stackrel{\nabla}{\longrightarrow}T^*S^n\otimes\mathbb{T}^{k-1}\cong_{\spin(n)}
\mathbb{T}^{k-2}\oplus\mathbb{T}^{k-1}\oplus\mathbb{T}^k\oplus\mathbb{Z}^{k-1},\quad 2\le k \le (n-2)/2. 
\end{equation*}
where $\mathbb{Z}^{k-1}\cong_{\text{Spin($n$)}}\mathbb{V}(\tfrac{5}{2},\underbrace{\tfrac{3}{2},\dots,\tfrac{3}{2}}_{k-2},\tfrac{1}{2},\dots\tfrac{1}{2})$.
Note also that the formal adjoint of $T_{k-1}$ is $T^*_{k-1}=\tilde{\delta}$. When $k=1$, $(T_0\u)_a=\N_a\u+\frac{1}{n}\g_a\dc\u$ is the twistor operator (\cite{BH:02}).
\begin{lem}
The second order operator $T_{k-1}T_{k-1}^*$ acts as a scalar
\begin{equation*}
\begin{array}{cl}
0 &\quad \text{ on }\cV_\e(j,1)\, \text{ and}\\ 
\displaystyle{\frac{(n-2k+1)(L^2-(n/2-k+1)^2)}{k(n-2k+2)}} &\quad \text{ on }\cV_\e(j,0)\, . 
\end{array}
\end{equation*}
\end{lem}
\begin{proof}
$T_{k-1}T^*_{k-1}$ clearly annihilates $\cV_\e(j,1)$ type. Assume that $T_{k-1}T^*_{k-1}\f=\l\f$ for $\f\in\cV_\e(j,0)$. Then $T^*_{k-1}T_{k-1}T^*_{k-1}\f=\l T^*_{k-1}\f$ and $T^*_{k-1}\f\in\cV_\e(j,1)$ over $\bbT^{k-1}$.	So we take $\y\in\cV_\e(j,1)$ over $\bbT^{k-1}$ and compute $T^*_{k-1}T_{k-1}\y$.
\begin{align*}
&T^*_{k-1}T_{k-1}\y_{a_1\cdots a_{k-1}}\\
&\quad =\frac{1}{k}\sum_{i=0}^{k-1}(-1)^{i+1}\left(\N^{a_0}\N_{a_i}\y_{a_0\cdots a_{i-1}a_{i+1}\cdots a_{k-1}}+\frac{1}{n-2k+2}\N^{a_0}\g_{a_i}\dc\y_{a_0\cdots a_{i-1}a_{i+1}\cdots a_{k-1}}\right)\\
&\quad =\frac{1}{k}\left(\N^*\N-\frac{1}{n-2k+2}\dc^2\right)\y_{a_1\cdots a_{k-1}}+\frac{1}{k}\underbrace{\sum_{i=1}^{k-1}(-1)^{i+1}\N^{a_0}\N_{a_i}\y_{a_0\cdots a_{i-1}a_{i+1}\cdots a_{k-1}}}_{A}\\
&\quad\quad +\frac{1}{k(n-2k+2)}\underbrace{\sum_{i=1}^{k-1}(-1)^{i+1}\N^{a_0}\g_{a_i}\dc\y_{a_0\cdots a_{i-1}a_{i+1}\cdots a_{k-1}}}_{B}.
\end{align*}
For $A$, we compute
\begin{align*}
&(-1)^{i+1}\N^{a_0}\N_{a_i}\y_{a_0\cdots a_{i-1}a_{i+1}\cdots a_{k-1}}=(-1)^{i+1}(\N_{a_i}\N^{a_0}+\cR^{a_0}{}_{a_i})\y_{a_0\cdots a_{i-1}a_{i+1}\cdots a_{k-1}}\\
&\quad =(-1)^{i+1}\cR^{a_0}{}_{a_i}\y_{a_0\cdots a_{i-1}a_{i+1}\cdots a_{k-1}}=(n-k+3/2)\y_{a_1\cdots a_{k-1}},
\end{align*}
where $\cR$ is the spin curvature (Theorem \ref{DB}). 
Similar computation shows that $B=0$. Thus, on $\cV_\e(j,1)$ over $\bbT^{k-1}$, $T^*_{k-1}T_{k-1}$ is
\begin{align*}
\frac{1}{k}\left(\N^*\N+\left(n-k+\frac{3}{2}\right)(k-1)-\frac{1}{n-2k+2}\dc^2\right).
\end{align*}
By Theorem \ref{DB}, this proves our claim.
\end{proof} 
Putting the above observations together, we have
\begin{thm} The differential operator $D_{2l+1,k}:\bbT^k\to\bbT^k$ defined by 
\begin{equation*}
D_{2l+1,k}:=\frac{1}{n-2k+2}P_k\prod_{i=1}^l \left(\frac{1}{(n-2k+2)^2}P_k^2-i^2\cdot\id-c_i\cdot T_{k-1}T^*_{k-1}\right),
\end{equation*}
where 
$$
c_i=\frac{16ki^2}{(n-2k+2)(n-2k+2-2i)(n-2k+2+2i)}
$$
is conformally invariant of order $2l+1$.
\end{thm}
\begin{proof} The operator $D_{2l+1,k}$ acts as a constant
$$
\begin{cases} J(J^2-1^2)\cdots (J^2-l^2) &\quad \text{on}\; \cV_\e(j,1)\\
\dfrac{n-2k-2l}{n-2k+2+2l}J(J^2-1^2)\cdots (J^2-l^2) &\quad \text{on}\; \cV_\e(j,0)
\end{cases}
$$
Thus by the theorem \ref{form}, $D_{2l+1,k}$ is a constant multiple of $Z_\e(l+1/2,j,q)$.
\end{proof}
\textbf{Case III: $\bf n$ even}\\
\\
Let $E=\eye\g^1(1-2\e(d\r)\i(\ptl\r))$ where $\g^1=\g(d\r)$ ((\ref{FTS})), $\e$ is the exterior multiplication, and $\i$ is the interior multiplication. $E$ changes chairality of the spinor, since $\g^1:\S_\pm\to\S_\mp$. It is readily verified that $E^2=\id$. And for $\F\in\bbT^k$, 
\begin{align*}
\g^j(E\Phi)_{ji_2\cdots i_k}&=\sqrt{-1}\g^j\g^1(\Phi_{ji_2\cdots i_k}-2\delta^1_j\Phi_{1i_2\cdots i_k})\\&=\sqrt{-1}(-\g^1\g^j-2g^{1j})(\Phi_{ji_2\cdots i_k}-2\delta^1_j\Phi_{1i_2\cdots i_k})\\
&=\sqrt{-1}(2\g^1\g^1\Phi_{1i_2\cdot i_k}-2g^{11}\Phi_{1i_2\cdot i_k}+4g^{11}\Phi_{1i_2\cdot i_k})\\
&=0 .
\end{align*}
Thus $E:\bbT^k_\pm\to\bbT^k_\mp$ with $E^2=\id$. On $\bbT^0_\pm=\S_\pm$,  $E=\eye\g^1$. We also compute
\begin{align*}
\mathcal{L}_Y E&=\sqrt{-1}(\underbrace{\mathcal{L}_Y \g^1}_{=0}\cdot (1-2\varepsilon(d\r)\iota(\partial \r))+\g^1\mathcal{L}_Y(1-2\varepsilon(d\r)\iota(\partial \r)))=\sqrt{-1}\g^1\mathcal{L}_Y(1-2\varepsilon(d\r)\iota(\partial \r))\\
&=-2\sqrt{-1}\g^1\mathcal{L}_Y\varepsilon(d\r)\iota(\partial \r))=-2\sqrt{-1}\g^1\{\varepsilon(d\r)\iota([Y,\partial\r])+\varepsilon(d(Y\r))\iota(\partial\r)\}\\
&=-2\sqrt{-1}\g^1(-\cos\r\,\varepsilon(d\r)\iota(\partial\r)+\cos\r\,\varepsilon(d\r)\iota(\partial\r))=0 .
\end{align*}
Thus the intertwining relation for the exchanged operator $EA$ is exactly the same as that of $A$ ((\ref{Int-rel})). 
\begin{equation*}\label{rel-1}
EA\left(\tilde{\cL}_X+\left(\frac{n}2-r\right)\omega\right)=
\left(\tilde{\cL}_X+\left(\frac{n}2+r\right)\omega\right)EA.
\end{equation*}

We first consider $EA:\S_\pm\to\S_\pm$. Let $\cV(j)=\cV(1/2+j,1/2,\ldots,1/2)$.
With respect to the following diagram
\begin{equation*}
\begin{array}{c}
\cV(j+1)\\
\ua \\
\boxed{\cV(j)}\\
\da\\
\cV(j-1)
\end{array}
\end{equation*}
we get
\begin{equation*}
\begin{array}{c}
(J+1/2+r)/(J+1/2-r)\\
\ua\\
\bullet\\
\da\\
(-J+1/2+r)/(-J+1/2-r)\\
\end{array}
\end{equation*}
where $J=n/2+j$. 
Choosing a normalization $\m_{{}_{\cV(0)}}=\e$, we have
\begin{thm} The unique spectral function $Z(r,j)$ on $\cV_\e(j)$ is up to normalization
\begin{equation*}
Z(r,j)=\cdot\frac{\G(J+\frac{1}{2}+r)\G(\frac{n}{2}+\frac{1}{2}-r)}{\G(J+\frac{1}{2}-r)\G(\frac{n}{2}+\frac{1}{2}+r)},\quad j=0,1,2,\ldots. 
\end{equation*}
\end{thm}
\begin{cor} (\cite{BO:06}) The differential operator $D_{2l+1}:\S_\pm\to\S_\mp$ defined by 
\begin{equation*}
D_{2l+1}:=\dc\cdot\prod_{p=1}^l (\dc^2-p^2):\S_\pm\to\S_\mp
\end{equation*}
is conformally invariant of order $2l+1$.
\end{cor}
\begin{proof} $E\dc\cdot\prod_{p=1}^l (\dc^2-p^2)$ is a constant multiple of $Z(l+1/2,j)$. So $D_{2l+1}:\S_\pm\to\S_\mp$ is a differential intertwinor as well.
\end{proof}
Next we consider $EA:\bbT^k_\pm\to\bbT^k_\pm$ for $k\ge 1$. Let $\cV(j,q)=\cV(3/2+j,3/2,\ldots,3/2,\underbrace{1/2+q}_{(k+1)^\text{st}},1/2,\ldots,1/2)$.
With respect to the diagram
\begin{equation*}
\begin{array}{ccc}
&&\cV(j+1,1)\\
&&\ua \\
\cV(j,0)&\leftarrow&\boxed{\cV(j,1)}\\
&&\da\\
&&\cV(j-1,1)\\
\end{array}
\end{equation*}
we get
\begin{equation*}
\begin{array}{ccc}
&&(J+\frac{1}{2}+r)/(J+\frac{1}{2}-r)\\
&&\ua\\
(n-2k+1-2r)/(n-2k+1+2r)&\leftarrow&\bullet\\
&&\da\\
&&(-J+\frac{1}{2}+r)/(-J+\frac{1}{2}-r)\\
\end{array}
\end{equation*}
where $J=n/2+1+j$.
Choosing a normalization $\m_{{}_{\cV(0,1)}}=1$, we have
\begin{thm}\label{form-1} The unique spectral function $Z(r,j,q)$ on $\cV(j,q)$ is up to normalization
\begin{align*}
&Z(r,j,q)=\frac{n-2k+1+2(2q-1)r}{n-2k+1+2r}\cdot\frac{\G(J+\frac{1}{2}+r)\G(\frac{n}{2}+\frac{3}{2}-r)}{\G(J+\frac{1}{2}-r)\G(\frac{n}{2}+\frac{3}{2}+r)},\\
&\quad q=0,1, \quad j=0,1,2,\ldots. 
\end{align*}
\end{thm}
The proof of Lemma \ref{HS2} also shows that $P_k^2$ is a constant multiple of $\dc^2$ on $\cV(j,1)$. Thus $P_k^2$ acts as 
$$
(n-2k+2q)^2\cdot J^2 \quad\text{ on }\cV(j,q), \quad q=0,1.
$$
\begin{thm} The differential operator $D_{2l+1,k}:\bbT^k_\pm\to\bbT^k_\mp$ defined by
\begin{equation*}
D_{2l+1,k}:=\frac{1}{n-2k+2}P_k\prod_{i=1}^l \left(\frac{1}{(n-2k+2)^2}P_k^2-i^2\cdot\id-c_i\cdot T_{k-1}T^*_{k-1}\right),
\end{equation*}
where 
$$
c_i=\frac{16ki^2}{(n-2k+2)(n-2k+2-2i)(n-2k+2+2i)}
$$
is conformally invariant of order $2l+1$.
\end{thm}
\begin{proof} The operator
$$
\frac{1}{n-2k+2}EP_k\prod_{i=1}^l \left(\frac{1}{(n-2k+2)^2}P_k^2-i^2\cdot\id-c_i\cdot T_{k-1}T^*_{k-1}\right)
$$ 
is a constant multiple of $Z(l+1/2,q)$ in Theorem \ref{form-1}. 
\end{proof}
To show an example of the theorem, let us take $k=1$ and $l=1$. Then we get $$
D_{3,1}=\frac{1}{n} P_1\left(\frac{1}{n^2} P_1^2-\id-\frac{16}{n(n-2)(n+2)}T_0T_0^*\right).
$$
The third order conformally invariant differential operator $S_3$ on the twistor bundle ($\bbT^1$ for $n$ odd and $\bbT^1_\pm$ for $n$ even) over a general curved manifold shown in \cite{Branson:01} is
\begin{equation*}
(S_3\f)_i=\left(\left\{\frac{n+2}{4n^2}P_1^3-\frac{4}{n(n-2)}T_0T^*_0\cdot P_1\right\}\f\right)_i+\text{LOT},
\end{equation*}
where
\begin{align*}
\text{LOT}=&-\frac{n+2}{4}J\g_i\N^j\f_j+V_i{}^j\g^k\N_k\f_j+(n+2)V_i{}^k\g_k\N^j\f_j+(n+1)V^{jk}\g_i\N_k\f_j\\
&-\frac{n(n+2)}{2}V^{jk}\g_k\N_j\f_i+(n-1)V^{jk}\g_k\N_i\f_j+V^{jl}\g_{ikl}\N^k\f_j+\frac{n}{2}(\N^j J)\g_i\f_j\\
&-\frac{n(n+2)}{4}(\N^j J)\g_j\f_i+n(\N^k V_i{}^j)\g_k\f_j.
\end{align*}
Here we used
$$
J=\frac{\text{Scal}}{2(n-1)},\quad V=\frac{r-Jg}{n-2},\quad \g_{ijk}=\g_{(i}\g_{j}\g_{k)},
$$
where Scal is the scalar curvature, $r$ is the  Ricci curvature, $g$ is the metric tensor, and $\g_{ijk}$ is the skew-symmetrization of $\g_i\g_j\g_k$ over $i,j,k$.

On the standard sphere, LOT simplifies to $-\dfrac{n+2}{4}P_1$.
Thus $S_3=\dfrac{n(n+2)}{4} D_{3,1}$.


\vspace{1cm}
\noindent
Department of Mathematics\\
University of North Dakota\\
Grand Forks ND 58202 USA\\
email: doojin.hong@und.edu

\end{document}